\documentclass [12pt]{article}

\usepackage {a4}
\usepackage {amsfonts}
\usepackage {amssymb,amsmath,amsthm}
\usepackage [utf8]{inputenc}
\usepackage [english]{babel}
\usepackage{enumerate}
\usepackage{graphicx}
\usepackage{fullpage}
\usepackage{bbm}
\usepackage{hyperref}

\newcommand{\R}{\mathbb{R}}

\newcommand{\C}{\mathbb{C}}
\newcommand{\Pp}{\mathbb{P}}
\newcommand{\E}{\mathbb{E}}

\newcommand{\abs}[1]{\left\vert #1 \right\vert}

\newcommand{\Var}{\mathrm{Var}}

\newcommand{\wt}{\mathrm{wt}}
\newcommand{\sgn}{\mathrm{sgn}}
\newcommand{\ordST}{\mathrm{ord}}

\newtheorem{theorem}{Theorem}[section]
\newtheorem{conjecture}[theorem]{Conjecture}
\newtheorem{definition}[theorem]{Definition}
\newtheorem{lemma}[theorem]{Lemma}

\newtheorem{proposition}[theorem]{Proposition}

\newtheorem{remark}[theorem]{Remark}
\theoremstyle{definition}

\author{Marcin Kotowski \and B{\'a}lint Vir{\'a}g}
\title{Tracy-Widom fluctuations in 2D random Schr{\"o}dinger operators}

\begin{document}
\maketitle

\begin{abstract}
We construct a random Schr{\"o}dinger operator on a subset of the hexagonal lattice and study its smallest positive eigenvalues. Using an asymptotic mapping, we relate them to the partition function of the directed polymer model on the square lattice. For a specific choice of the edge weight distribution, we obtain a model known as the log-Gamma polymer, which is integrable. Recent results about the fluctuations of free energy for the log-Gamma polymer allow us to prove Tracy-Widom type fluctuations for the smallest eigenvalue of the random Schr{\"o}dinger operator. We also relate the distribution of its $k$ smallest positive eigenvalues to the nonintersecting partition functions of order $k$.
\end{abstract}


\begin{figure}[h!]
  \centering
\includegraphics[scale=0.2]{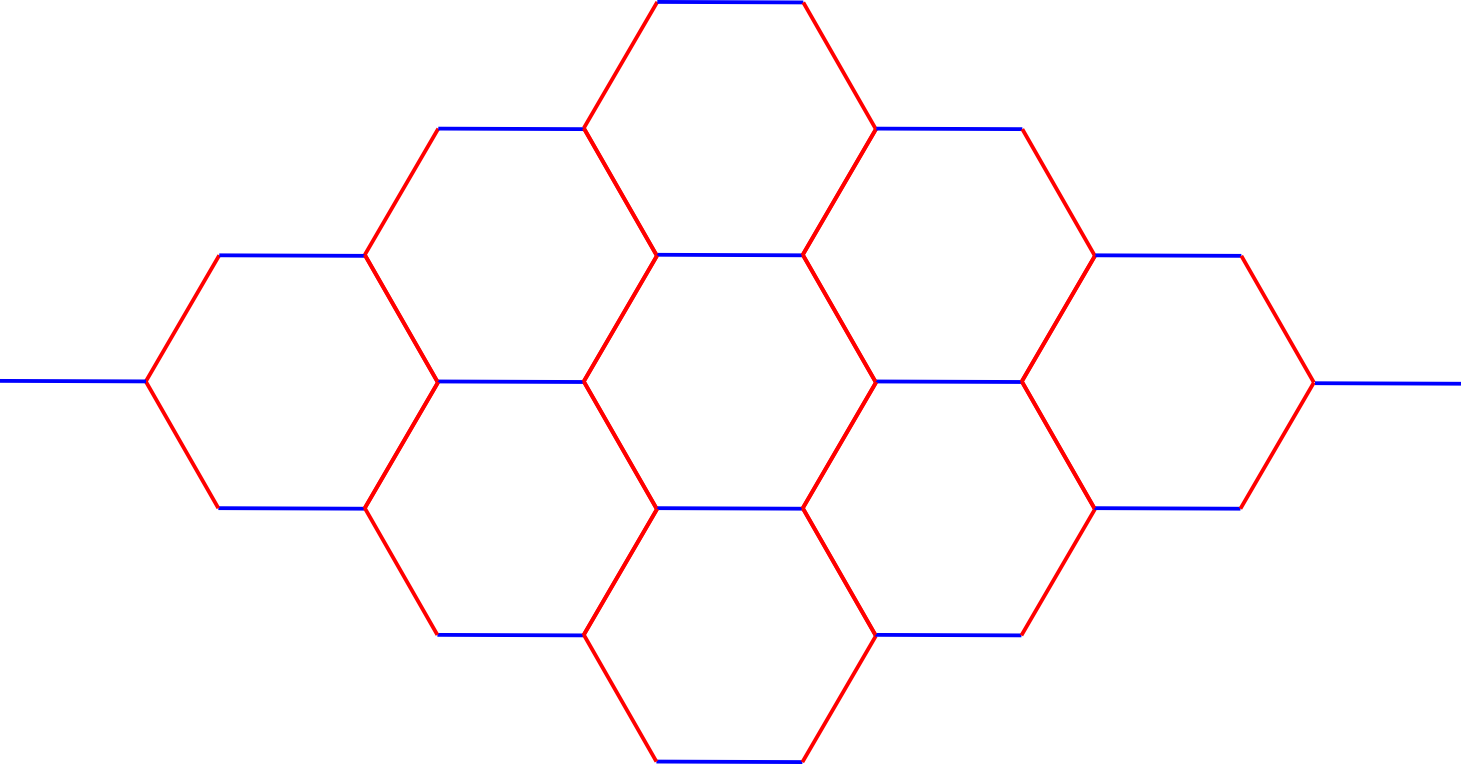}
  \caption{The lattice.}
\label{fig:hex}
\end{figure}

\section{Introduction}\label{sec:intro}

In this paper, we study the correspondence between certain random Schr{\"o}dinger operators defined on a subset of the 2D hexagonal lattice and a statistical physics model known as the directed log-Gamma polymer model. The directed log-Gamma polymer on a square lattice is obtained by putting random weights on the vertices of the lattice, drawn from the inverse Gamma distribution, and considering up-right paths connecting the opposite corners of the square, where each path is weighted by the product of its vertices. One is then interested in various statistical properties of such paths. The model has recently attracted considerable attention \cite{seppa},\cite{log-gamma}, \cite{jeremy}, as it is integrable, i.e. allows explicit computations.

We construct a 2D random Schr{\"o}dinger operator $H$ and a mapping which maps its eigenvalues onto certain quantities in the directed polymer model, called the partition functions. Using results about the fluctuations of free energy for log-Gamma polymers \cite{jeremy}, we prove Tracy--Widom GUE fluctuations for the smallest positive eigenvalue of $H$ (Theorem \ref{th:main}). To our knowledge this is the first known example of such fluctuations for a random Schr{\"o}dinger operator. Moreover, we provide a description of higher eigenvalues in terms of partition functions related to non-intersecting paths. Such objects arise naturally in the technique knows as geometric Robinson-Schoensted-Knuth correspondence \cite{tropical}.

We consider a random Schr{\"o}dinger operator defined on a hexagonal lattice in the shape of a rhombus. Formally, let $G_n$ be a subset of the hexagonal lattice consisting of $2n-1$ levels, with level $k$, for $k=0,\dots,2n-2$, containing $\min\{k,2n-k-2\}$ hexagons. The first and last level contain only a single edge. An example of such lattice for $n=4$ is shown in figure \ref{fig:hex}. We will call horizontal edges blue and the remaining edges red.

We consider edges equipped with random real-valued weights, where the weight of an edge $e$ is denoted by $w_e$. The random Schr{\"o}dinger operator $H_n$, acting on functions $f: G_n \rightarrow \R$, is the weighted adjacency operator on $G_n$:
\[
(H_n f)(v) = \sum_{e = (v,w)}w_e f(w)
\]
where the sum is over all edges adjacent to $v$.


We consider two models defined on the lattice $G_n$:
\begin{enumerate}
\item {\bf (i.i.d. model)} All edge weights are drawn independently at random from some distribution $X$ that is nonzero almost surely and satisfies $\E \log\abs{X} \geq 0$ and $\E e^{-t\log \abs{X}}, \E e^{t\log \abs{X}} < \infty$ for some $t > 0$
\item {\bf (mixed model)} Pick some parameter $\gamma > 0$. The red edges are given weight $1$. Each blue edge is independently assigned a weight drawn from the Gamma distribution $\Gamma(\gamma,1)$.
\end{enumerate}

In the mixed model, we prove the following theorem about the smallest positive eigenvalue of $H_n$:
\begin{theorem}\label{th:main}
Let $\lambda_n$ be the smallest positive eigenvalue of $H_n$ in the mixed model with parameter $\gamma$. For all $\gamma < \gamma^{\ast}$ we have as $n \rightarrow \infty$:
\[
\Pp\left(\frac{-\log \lambda_n - \bar{f}_{\gamma}n}{n^{1/3}} \leq r\right) \rightarrow F_{\mathrm{GUE}}\left(\left(\frac{\bar{g}_{\gamma}}{2}\right)^{-1/3}r\right)
\]
where $\gamma^{\ast} = 1.461632...$ is the unique positive real root of the digamma function $\Psi$, $F_{GUE}$ is the GUE Tracy-Widom distribution function, $\bar{f}_{\gamma} = -2 \Psi(\gamma / 2)$ and $\bar{g}_{\gamma} = - 2 \Psi''(\gamma / 2)$.
\end{theorem}

The i.i.d.\ model is arguably more natural and we expect the theorem to hold also in that case:
\begin{conjecture}\label{conj:iid}
Theorem \ref{th:main} holds also in the i.i.d.\ model for appropriate choice of constants.
\end{conjecture}

In both models, we make the following conjecture generalizing Tracy--Widom fluctuations also to higher eigenvalues:
\begin{conjecture}\label{conj:airy}
Let $\lambda_{n,k}$ be the $k$th smallest positive eigenvalue of $H_n$. Let
$$
\alpha_{n,k}=\frac{-\log \lambda_{n,k} -\bar f_\gamma n}{(n\bar g_\gamma/2)^{1/3}}
$$
as in Theorem \ref{th:main}.
Then for any $k$, as $n\to\infty$, the tuple $(\alpha_{n,1},\ldots , \alpha_{n,k})$ converges in distribution to the top $k$ points of the Airy point process.
\end{conjecture}

By considering submatrices of $H_n$, this conjecture can be extended to multiple space-time
values of the (conjectured) scaling limit of last passage percolation. In particular, it should be possible to get the Airy sheet \cite{airy} as a limit in this model. This is in contrast with standard random matrix eigenvalue models, for which the Airy sheet is not expected to arise as a limit. 

We make a step toward Conjecture \ref{conj:airy} by proving that in the mixed model, the product of the bottom $k$ eigenvalues is related, up to order $n^{1/3}$, to the partition functions for $k$-tuples of non-intersecting paths. Such objects appear naturally while studying exact formulas related to the geometric RSK correspondence \cite{tropical}.  The theorem holds also for more general models, see Theorem \ref{th:mixed-topk}.

\begin{theorem}\label{th:mixed-topk-intro}
Let $\lambda_n, \dots, \lambda_{n-k+1}$ be the $k$ smallest positive eigenvalue of $H_n$ in the mixed model with parameter $\gamma$. For any fixed $k \geq 1$, let $Z^{(k)}_{n}$ be the non-intersecting partition function of order $k$ for the square lattice corresponding to the mixed model (Definition \ref{def:nonintersect-Z}). Then for any $\gamma < \gamma^{\ast}$ and any $\delta > 0$:
\[
\Pp\left(n^{-1/3}\abs{-\log\prod_{i=1}^{k}\lambda_{n-i+1} - \log \abs{Z^{(k)}_n}} > \delta \right) \rightarrow 0
\]
where $\gamma^{\ast}$ is the unique positive root of the digamma function.
\end{theorem}

We end this section with an outline of how the theorems are proved. In Section \ref{sec:eigen}, we prove that the eigenvalues of the operator are equal to the square roots of the singular values of the directed weighted square lattice. These, in turn, happen to be related to the partition functions of the polymer model on the lattice (Theorem \ref{th:singular}). Using this connection, in Section \ref{sec:intersect} we proceed to prove Theorem \ref{th:mixed-topk-intro} using a technical combinatorial lemma whose proof is contained in Section \ref{sec:comb}. Then, in Section \ref{sec:fluctuations}, we prove Theorem \ref{th:main} by exploiting known resuts about the fluctuations of the partition functions for the log-Gamma polymer.

\section{Eigenvalues and polymers}\label{sec:eigen}

The results in this section are deterministic -- we introduce the probabilistic part of the analysis in Section \ref{sec:prob}. In order to study the eigenvalues of a random Schr{\"o}dinger operator on a graph $G$, we first prove a lemma allowing us to study instead singular values of a certain directed graph derived from $G$.

\begin{lemma}\label{lm:dual}
Let $G = (V, E)$ be a weighted bipartite graph on $2n$ vertices with bipartition $V = A \sqcup B$. Let $w_e$ denote the weight of the edge $e$. Suppose that $G$ admits a perfect matching $M \subseteq E$ with edges $e_i = (a_i, b_i), a_i \in A, b_i \in B, i = 1, \dots, n$. Let $\widetilde{G}$ be a weighted directed graph on $n$ vertices, with vertex set $M$ and with edges defined as follows. For each $e_i \in M$, we have a loop $(e_i,e_i)$ with weight $w_{e_{i}}$. For each edge $f = (a_i,b_j) \notin M$, we have a directed edge $(e_i, e_j)$ with weight $w_f$.

Let $A$ be the adjacency matrix of $G$ and let $\widetilde{A}$ be the adjacency matrix of $\widetilde{G}$. Then the eigenvalues $\lambda_i$ of $A$ are equal to $\pm \sigma_i$, where $\sigma_i$ are the singular values of $\widetilde{A}$.
\end{lemma}

\begin{proof}
Let $A = (a_1, \dots, a_n), B = (b_1, \dots, b_n)$, ordered arbitrarily. Let us index the rows and columns of $A$ with $(a_1, \dots, a_n, b_1, \dots, b_n)$. Then $A$ has the block form:
\[
A = 
\begin{pmatrix}
0 & \widetilde{A} \\
\widetilde{A}^T & 0
\end{pmatrix}
\]
Indeed, each edge $(a_i, b_i)$ in $G$ corresponds to the edge $(e_i, e_i)$ in $\widetilde{G}$, giving the diagonal entries of $\widetilde{A}$. Each edge $(a_i, b_j)$ for $i \neq j$ corresponds to an edge $(e_i, e_j)$ in $\widetilde{G}$, giving the off-diagonal entries. Clearly, the eigenvalues of $A$ are equal to $\pm$ the square roots of eigenvalues of $\widetilde{A}\widetilde{A}^T$, which are simply the singular values of $\widetilde{A}$.
\end{proof}

We now construct a general mapping between singular values of a directed graph $G$ and partition functions of the polymer model on the same graph. The results are stated in generality, but will be used for directed graphs derived from the particular lattice $G_n$ described in Section \ref{sec:intro}.

Let $G$ be a directed acyclic weighted graph on $n$ vertices and let $A$ denote its adjacency matrix. Assume that every vertex has a loop with nonzero weight. This implies that $A$ is invertible. Indeed, consider the set of vertices with no incoming edges, which is nonempty since the graph is acyclic. Since the loop weights are nonzero, the equation $Af=0$ implies that $f=0$ at such vertices. We can then remove them and repeat until there are no vertices left, proving that $f \equiv 0$.

For $v,w \in G$, a path $\pi$ from $v$ to $w$ is defined to be a sequence of edges connecting vertices $(v=u_{1} \rightarrow u_2 \rightarrow \dots \rightarrow u_{n}=w)$, where none of the edges are loops. We allow a path of length zero connecting a vertex $v$ to itself. Let $\Pi(v,w)$ denote the set of all paths from $v$ to $w$. We will say that a vertex $v$ precedes $w$ if there is a positive length path from $v$ to $w$.

We define new weights on vertices and edges of $G$ in the following way. For a vertex $u$ we put $w_u = \frac{1}{A_{u,u}}$ and for an edge $e=(u,v)$ we put $w_{u,v} = -A_{u,v}$. For a path $\pi  = (u_1 \rightarrow \dots \rightarrow u_n)$ let its weight $\wt(\pi)$ be defined as: 
\begin{equation}\label{eq:wt}
\wt(\pi) := \prod_{i=1}^{n-1}w_{u_{i},u_{i+1}}\prod_{i=1}^{n} w_{u_{i}} 
\end{equation}
Note that the weight of an empty path from $u$ to itself is $w_u$.

\begin{definition}\label{def-partitionfcn}
Fix any $k \in \{1,\dots,n\}$ and two sequences of distinct vertices $S =(u_1, \dots, u_k)$, $V = (v_1,\dots,v_k)$. Consider $k$-tuples of vertex-disjoint paths $\pi=(\pi_1,\dots,\pi_k)$, with $\pi_i$ connecting $u_i$ and $v_{\sigma(i)}$ for some permutation $\sigma$. Denote this permutation by $\sigma(\pi)$. We define:
\[
Z^{(k)}_{S,T} := \sum_{\pi=(\pi_1,\dots,\pi_k)}\sgn(\sigma(\pi))\prod_{i=1}^{k}\wt(\pi_i)
\]
For $k=1$ we will simply write:
\[
Z_{u,v} = \sum_{\pi: v\rightarrow w}\wt(\pi)
\]
\end{definition}
We put $Z_{u,v}$ equal zero if there are no paths from $u$ to $v$.


For $u \in G$, let $f_u$ denote the function defined on the vertices of $G$ by $f_u(v) = Z_{u,v}$. In particular, $f_u(v)=0$ if $v$ precedes $u$ and $f_u(u)=w_u$. Let $\delta_u$ be the function equal to $1$ on $u$ and $0$ otherwise.

\begin{proposition}\label{prop:fu}
The functions $f_u$ satisfy $Af_u = \delta_u$.
\end{proposition}
\begin{proof}
We clearly have:
\[
(Af_v)(v) = A_{v,v}f_v(v) = 1
\]
For $v \neq w$, we have:
\begin{align*}
&(Af_v)(w) = \sum_{u \rightarrow w}A_{u,w}f_v(u) + A_{w,w}f_v(w) =
\sum_{u \rightarrow w}A_{u,w}\sum_{\pi \in \Pi(v,u)}\wt(\pi) + A_{w,w}\sum_{\pi \in \Pi(v,w)}\wt(\pi) = \\
&-A_{w,w}\sum_{\sigma \in \Pi(v,w)}\wt(\sigma) + A_{w,w}\sum_{\pi \in \Pi(v,w)}\wt(\pi) = 0
\end{align*}
\end{proof}

The quantities $Z^{(k)}_{S,T}$ can be related to $f_u$ using the well known Lindstrom-Gessel-Viennot formula \cite{gv} for expressing sums over non-intersecting paths as determinants:
\begin{proposition}\label{prop:gsv}
For $S=(u_1,\dots,u_k),T=(v_1,\dots,v_k)$ we have:
\begin{equation}\label{eq:gsv}
Z^{(k)}_{S,T} = \det(f_{u_{i}}(v_{j}))_{i,j=1}^{k}
\end{equation}
\end{proposition}

\begin{proof}
The standard Lindstrom-Gessel-Viennot formula is usually formulated with weights only on the edges. To obtain \ref{eq:gsv} in the general case, consider a graph $G'$ where for an edge $(u,v)$ we put $w'_{u,v} = w_{u,v} w_u $ and $w'_{u}=1$. By applying the standard Lindstrom-Gessel-Viennot formula to $G'$ we obtain $Z^{(k)'}_{S,T} = \det(f'_{u_{i}}(v_{j}))_{i,j=1}^{k}$. The proof follows by noting that $Z^{(k)}_{S,T} = Z^{(k)'}_{S,T} \cdot \prod_{v \in T}w_v$ and $f_u(v) = f'_u(v) \cdot w_v$.
\end{proof}

Let $\sigma_1 \geq \sigma_2 \geq \dots \geq \sigma_n$ denote the singular values of $A^{-1}$. 

\begin{theorem}\label{th:singular}
For any $k=1,\dots,n$, we have:
\[
\max_{S,T}\abs{Z^{(k)}_{S,T}} \leq \prod_{i=1}^{k}\sigma_i(A^{-1}) \leq \binom{n}{k}^2 \cdot \max_{S,T}\abs{Z^{(k)}_{S,T}}
\]
where the maximum ranges over all pairs of sequences of distinct vertices $S=(u_1,\dots,u_k), T=(v_1,\dots,v_k)$.
\end{theorem}

\begin{proof}
We first use the following formula for the product of the singular values \cite{hogben}:
\begin{equation}\label{eq:sigma-uav}
\prod_{i=1}^{k}\sigma_i(A^{-1}) = \max\{\abs{\det(U^{\ast}A^{-1}V)} : U,V \in \C^{n\times k}, UU^{\ast}=VV^{\ast}=I_k\}
\end{equation}

For a sequence of distinct vertices $S$ of size $k$ and a matrix $B \in \C^{n\times k}$, let $B_{S}$ denote the submatrix obtained by taking rows with indices corresponding to $S$. With this notation $I_S$ is the matrix having columns equal to $\delta_s$ for $s \in S$, i.e. the coordinate vectors corresponding to vertices in $S$. We have $B_S = I_S^{\ast} B$.

For the lower bound, for any $S=(u_1,\dots,u_k), T=(v_1,\dots,v_k)$ we plug $U=I_S, V=I_T$ into \eqref{eq:sigma-uav}. Note that by Proposition \ref{prop:fu}, the matrix $A^{-1}$ expressed in the basis consisting of $\delta_u$ has the functions $f_u$ as its columns. Thus, by Proposition \ref{prop:gsv} we have $\det((I_S)^{\ast} A^{-1}I_T) = Z^{(k)}_{S,T}$, from which the lower bound follows.

For the upper bound, for any $U,V$ we use the Cauchy-Binet formula twice:
\begin{align*}
&\det(U^{\ast}A^{-1}V) = \sum_{S}\det(U^{\ast}_{S})\det((A^{-1}V)_{S}) =
\sum_{S}\det(U^{\ast}I_S)\det(I_S^{\ast} A^{-1}V) = \\
& \sum_{S,T}\det(U^{\ast}I_S) \cdot \det(I_S^{\ast} A^{-1}I_T) \cdot \det(I_T^{\ast}V)
\end{align*}
Clearly, we have $\abs{\det(U^{\ast}I_S)},\abs{\det(I_T^{\ast}V)} \leq 1$, so:
\[
\max_{U,V}\abs{\det(U^{\ast}A^{-1}V)} \leq \binom{n}{k}^2 \cdot\max_{S,T}\abs{\det(I_S^{\ast} A^{-1}I_T)} = \binom{n}{k}^2 \cdot \abs{\max_{S,T}Z_{S,T}}
\]
\end{proof}

We will now apply the construction above to the hexagonal lattice $G_n$ from the previous section. In the case of $G_n$, the perfect matching in Lemma \ref{lm:dual} consists of blue edges. The corresponding directed graph $\widetilde{G}_n$ is a directed square lattice with a loop added to each vertex. Both lattices are shown in Figure \ref{fig:lattices}. 

\begin{figure}[ht!]
  \centering
\includegraphics[scale=0.2]{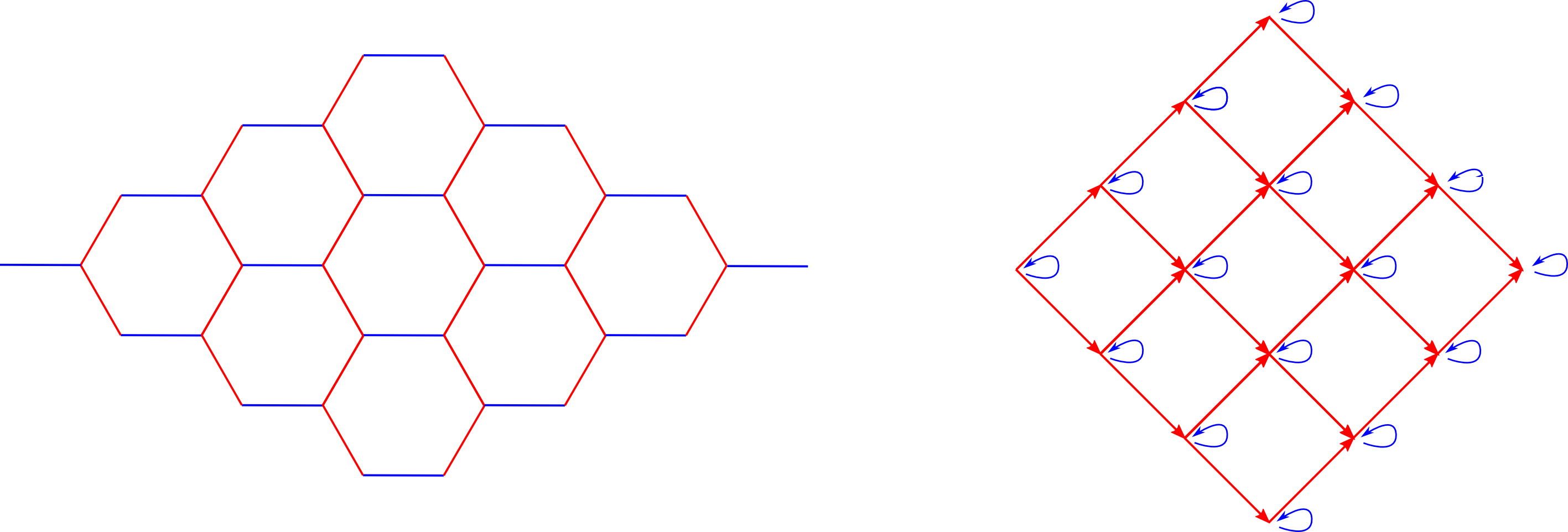}
  \caption{The lattice $G_n$ and the corresponding directed lattice $\widetilde{G_n}$.}
\label{fig:lattices}
\end{figure}

\begin{remark}\label{rm:mixed-model}
In the mixed model, all edges of the directed square lattice have weights $-1$. Since each $A_{u,u}$ was drawn independently from the Gamma distribution, each loop $u$ has a weight $w_u = \frac{1}{A_{u,u}}$ drawn independently from the inverse Gamma distribution $\Gamma^{-1}(\gamma)$ (Definition \ref{def:gamma}).
\end{remark}

\begin{remark}\label{rm:iid-model}
In the i.i.d. model, all edges of the hexagonal lattice have i.i.d. weights distributed as some random variable $X$. This implies that on the directed square lattice each edge weight $w_{u,v}=-A_{u,v}$ is distributed as $-X$ and each vertex weight $w_u = \frac{1}{A_{u,u}}$ is distributed as $\frac{1}{X}$. 
\end{remark}

\section{Probabilistic results}\label{sec:prob}

We now introduce the probabilistic part of the analysis for the square lattice with random edge and vertex weights. We do not require edge weights to be independent, only that for each path its edges are independent. The square lattice considered is the one from Figure \ref{fig:lattices} rotated $45$ degrees counterclockwise, so that the lower left corner is the point $(1,1)$ and the upper right corner is the point $(n,n)$. We fix some $k \geq 1$ and let $S_0=((1,1),(1,2),\dots,(1,k)), T_0=((n,n-k+1),(n,n-k+2),\dots,(n,n))$. 

For an edge $e$ let $X_e := \log \abs{w_{e}}$ and for a vertex $u$ let $X_u := \log \abs{w_{u}}$. We assume that $\E \log\abs{X_e}, \E \log\abs{X_u}\geq 0$. 
\begin{definition}\label{def:nonintersect-Z}
For the directed square lattice from $(1,1)$ to $(n,n)$, the non-intersecting partition function of order $k$ is defined as:
\[
Z^{(k)}_{n} := Z^{(k)}_{S_{0},T_{0}} = \sum_{\pi=(\pi_1,\dots,\pi_k)}\prod_{i=1}^{k}\wt(\pi_i)
\]
where the summation is over all tuples of $k$ vertex disjoint paths, with $\pi_i$ connecting $(1,i)$ to $(n,n-k+i)$.
\end{definition}

In this section, we prove two results. First, in Section \ref{sec:intersect}, for the case $\E X_e >0$ we show that up to order $n^{1/3}$ the product of $k$ top singular values of $A^{-1}_n$ is with high probability close the quantities $Z^{(k)}_n$ (Theorem \ref{th:mixed-topk}). By known results about fluctuations of the polymer partition function, this then implies (Theorem \ref{th:tracy--Widom}) Tracy--Widom fluctuations of the smallest singular value of $A_n$, for the weights drawn from the inverse Gamma distribution.

We shall rely on the following large deviation inequality \cite{durrett}. Pick a path $\pi$ of length at most $n$. Assuming that $\E e^{-tX_e} < \infty$ for some $t>0$, since $\E X_e \geq 0$ we have:
\begin{equation}\label{eq:ld1}
\Pp\left(\sum_{e \in \pi}X_e < -\delta n\right) \leq e^{-I(\delta) n} 
\end{equation}
for some rate function $I$.
Whenever we say that an event holds with high probability (w.h.p.), it will mean that the probability that it does not hold is superpolynomially small in $n$.

\subsection{Eigenvalues and non-intersecting partition functions}\label{sec:intersect}

The goal of this section is the proof of Proposition \ref{prop:intersect-aux}, which combined with the results from Section \ref{sec:eigen} implies  Theorem \ref{th:mixed-topk-intro} and Theorem \ref{th:mixed-topk}.

It will be convenient to work in the case when all vertex weights are equal to $1$. The proposition below shows that if we study the behavior of partition functions up to order $n^{1/3}$, we can do so without loss of generality. 
\begin{proposition}\label{prop:only-edges}
Suppose that for all vertices $u$ we have $\E e^{-tX_u} < \infty, \ \E e^{tX_u} < \infty$ for some $t > 0$ and likewise for edges. For an edge $(u,v)$ put $w'_{u,v} = w_{u,v} \cdot w_v$ and put $w'_u=1$ for all vertices $u$. Note that the primed edge weights are not independent, but they are still independent along every path. For any $k\geq 1, \delta >0$ and all $S,T$ we have:
\[
\Pp(n^{-1/3}\abs{\log \abs{Z^{(k)}_{S,T}} - \log \abs{Z^{(k)'}_{S,T}}} > \delta) \rightarrow 0
\]
\end{proposition}
\begin{proof}
We have $Z^{(k)}_{S,T} = Z^{(k)'}_{S,T}\cdot \prod_{u \in S}w_u$, so $\log \abs{Z^{(k)}_{S,T}} = \log \abs{Z^{(k)'}_{S,T}} + \sum_{u \in S}\log \abs{w_u}$. Since $\E e^{-tX_u},  \E e^{tX_u} < \infty$ for some $t > 0$, by Markov inequality for each $w_u$ we have:
\[\Pp(\abs{\log \abs{w_u}} > \delta n^{1/3}) < C e^{-t\delta n^{1/3}}\]
for some constant $C>0$. By union bounding over polynomially many choices of $S$ we can assume that w.h.p. for all choices of $S$ we have $\sum_{u \in S}\abs{\log \abs{w_u}} < Ck\delta n^{1/3}$, which finishes the proof. 
\end{proof}

Below we assume that all vertex weights are equal to $1$. For a pair of sequences $S=(s_1,\dots,s_k),T=(t_1,\dots,t_k)$ and a permutation $\sigma$, we let $\Pi_{\sigma,S,T}$ denote the set of all tuples of  paths $\pi=(\pi_1,\dots,\pi_k)$ with $\pi_i$ connecting $s_i$ to $t_{\sigma(i)}$. We let $\Pi^{n.i.}_{\sigma,S,T}$ denote the set of all such tuples with paths $\pi_i$ non-intersecting. For a tuple $\pi$ we let $E(\pi)$ denote the set of edges used by paths in $\pi$ (if an edge is used by multiple paths we count it once).  Recall that $S_0 = ((1,1),\dots,(1,k))$ and $T_0 = ((n,n-k+1),\dots,(n,n))$. The set of all non-intersecting tuples 
contributing to $Z^{(k)}_n$ is $\Pi^{n.i.}_{\mathrm{id},S_0,T_0}$. 

For $\pi \in \Pi_{\sigma,S,T}$, let $\wt(\pi) :=\prod_{i=1}^{k}\wt(\pi_i)$. Recall that:
\[
Z^{(k)}_{S,T} = \sum_{\sigma}\sgn(\sigma)Z_{\sigma,S,T}
\]
where:
\[
Z_{\sigma,S,T} := \sum_{\pi \in \Pi^{n.i.}_{\sigma,S,T}}\wt(\pi)
\]

The proof of Proposition \ref{prop:intersect-aux} will follow from the lemma below, which is purely combinatorial and whose proof we defer to Section \ref{sec:comb}. The lemma roughly says that any nonintersecting tuple connecting $S$ to $T$ can be modified into a nonintersecting tuple connecting $S_0$ to $T_0$ while removing only a constant number of edges and adding a constant number of path segments.


\begin{lemma}\label{lm:comb-everything}
There exists a constant $C$ depending only on $k$ such that for any $\sigma, S, T$ with $S \neq S_0$ or $T \neq T_0$ there exists a set of paths $\mathcal{P}$ which has size at most $C \cdot n^C$ and satisfies the following property. For every $\pi \in \Pi^{n.i.}_{\sigma,S,T}$ there exist a tuple $\pi' \in \Pi^{n.i.}_{\mathrm{id},S_0,T_0}$ such that $\abs{E(\pi) \backslash E(\pi')} \leq C$ and $E(\pi)\triangle E(\pi')$ is a union of paths whose number is at most $C$ and which are all elements of $\mathcal{P}$.
\end{lemma}

We note that the lemma is obvious in the case $k=1$, since it suffices to connect $S=\{s\}$ to $(1,1)$ and $T=\{t\}$ to $(n,n)$ with any two fixed paths.

\begin{proposition}\label{prop:intersect-aux}
For any fixed $k \geq 1$ and $\delta > 0$, if $\E X_e > 0$ and $\E e^{-tX_e} < \infty, \ \E e^{tX_e} < \infty$ for some $t > 0$, we have:
\[
\Pp\left(\frac{1}{n^{1/3}}\abs{\log \abs{Z^{(k)}_n} - \max_{S,T}\abs{\log Z^{(k)}_{S,T}}} > \delta \right) \rightarrow 0
\]
\end{proposition}

\begin{proof}
We need to prove that for any $k\geq 1$  and $\delta > 0$, we have w.h.p. for some global constant $D$ depending only on $k$:
\[
\abs{Z^{(k)}_n} \geq \max_{S,T}\abs{Z^{(k)}_{S,T}}\cdot e^{-D\delta n^{1/3}}
\]
Since:
\[
Z^{(k)}_{S,T} = \sum_{\sigma}\sgn(\sigma)Z_{\sigma,S,T}
\]
we have:
\[
\abs{Z^{(k)}_{S,T}} \leq k! \cdot \max_{\sigma}\abs{Z_{\sigma,S,T}}
\]
Thus, it suffices to prove that with high probability for all $\sigma, S, T$ we have:
\begin{equation}\label{eq:intersect-main}
\abs{Z^{(k)}_n} \geq \abs{Z_{\sigma,S,T}}\cdot e^{-D\delta n^{1/3}}
\end{equation}
Consider a tuple $\pi \in \Pi^{n.i.}_{\sigma,S,T}$ contributing to $Z_{\sigma,S,T}$. By Lemma \ref{lm:comb-everything} there exists $\pi' \in \Pi^{n.i.}_{\mathrm{id},S_0,T_0}$ such that $\abs{E(\pi) \backslash E(\pi')} \leq C$ and $E(\pi)\triangle E(\pi')$ is a union of at most $C$ paths which are elements of $\mathcal{P}$.

Let $A$ be the event that all edges of the lattice have weights at most $e^{t\delta n^{1/3}}$ and let $B$ the event that all paths in $\mathcal{P}$ have weights at least $e^{-\delta n^{1/3}}$. Note that both the number of edges in the lattice and the number of paths in $\mathcal{P}$ are polynomial in $n$. The large deviation inequality \ref{eq:ld1} together with union bound over a polynomial size family of events guarantees that $B$ holds with high probability. Likewise, by Markov inequality and union bound over all edges of the lattice the event $A$ also holds with high probability. Note that the events $B$ depends only on $\mathcal{P}$, which depends only on $\sigma,S,T$ and not on the tuple $\pi$.

Let $E(\pi)\triangle E(\pi') = \cup_{i=1}^{m}\pi_{i}$, where each path $\pi_i$ belongs to $\mathcal{P}$ and $m \leq C$. Since $A$ and $B$ hold w.h.p., we have with high probability:
\begin{equation}\label{eq:wtpi}
\wt(\pi') \geq \wt(\pi)\cdot e^{-(tC+C)\delta n^{1/3}}
\end{equation}



We now need to sum equation \eqref{eq:wtpi} over all paths $\pi \in \Pi^{n.i.}_{\sigma,S,T}$. The map taking $\pi$ to $\pi'$ need not be injective. However, note that $E(\pi) \triangle E(\pi')$ is a union of at most $C$ paths, all of which lie inside $\mathcal{P}$, which has size at most $C \cdot n^{C}$. Thus, each $\pi'$ will have at most polynomially many preimages. Thus, summation of equation \eqref{eq:wtpi} over all possible $\pi \in \Pi^{n.i.}_{\sigma,S,T}$ proves the desired inequality \eqref{eq:intersect-main}.
\end{proof}


\begin{remark}\label{rm:meanzero-intersect}
In the case $\E X_e = 0$, the logarithmic weight of a typical path is of order $\sqrt{n}$ and the proof does not apply. However, the same proof can be used to obtain a weaker statement, namely, for any $\varepsilon > 0$:
\[
\Pp\left(\frac{1}{n^{1/2 + \varepsilon}}\abs{\log \abs{Z^{(k)}_n} - \max_{S,T}\abs{\log Z^{(k)}_{S,T}}} > \delta \right) \rightarrow 0
\]
\end{remark}

By combining Proposition \ref{prop:intersect-aux} and Theorem \ref{th:singular}, we arrive at the following theorem.

\begin{theorem}\label{th:mixed-topk}
For any fixed $k \geq 1$ and $\delta > 0$, if $\E X_e > 0$ and $\E e^{-tX_e} < \infty, \ \E e^{tX_e} < \infty$ for some $t > 0$, we have: 
\[
\Pp\left(n^{-1/3}\abs{\log\prod_{i=1}^{k}\sigma_i(A^{-1}) - \log \abs{Z^{(k)}_n}} > \delta \right) \rightarrow 0
\]
\end{theorem}

We now proceed to prove Theorem \ref{th:mixed-topk-intro}. To this end, let us first note the following properties of the inverse Gamma distribution.

\begin{definition}\label{def:gamma}
A random variable $X$ has inverse Gamma distribution with parameter $\gamma > 0$, denoted $\Gamma^{-1}(\gamma)$, if its probability distribution is supported on positive reals with density:
\[
\Pp(X \in dx) = \frac{1}{\Gamma(\gamma)}x^{-\gamma-1}\exp\left(-\frac{1}{x}\right)dx
\]
\end{definition}

\begin{remark}\label{rm:gamma}
Let $X \sim \Gamma^{-1}(\gamma)$ and let $\Psi$ be the digamma function. Then $\E\log X = - \Psi(\gamma)$ and $\Var\log X = \Psi'(\gamma)$. In particular, from the properties of the digamma function, if we let $\gamma^{\ast}=1.461632...$ to be the unique real positive root of $\Psi$, for all $\gamma < \gamma^{\ast}$ we have $\E\log X > 0$. Also, for small enough $t>0$ we have $\E e^{-t\log X} < \infty, \E e^{t\log X} < \infty$.
\end{remark}

\begin{proof}[Proof of Theorem \ref{th:mixed-topk-intro}]
By Remark \ref{rm:mixed-model}, in the mixed model the dual graph corresponds to the directed square lattice with inverse Gamma vertex weights. By Proposition \ref{prop:only-edges}, if instead we put the weights on the edges, the difference between the partition functions of the vertex weighted model and the edge weighted model, scaled by $n^{-1/3}$, converges to zero in probability. By Remark \ref{rm:gamma}, for $\gamma < \gamma^{\ast}$ the inverse Gamma logarithmic edge weights satisfy the assumptions of Theorem \ref{th:mixed-topk}, so it holds also for the vertex weighted case. The proof follows by invoking Lemma \ref{lm:dual}. 
\end{proof}

\subsection{Fluctuations of the smallest eigenvalue in the exactly solvable case}\label{sec:fluctuations}

We now establish the fluctuations of the smallest eigenvalue in the mixed model (Theorem \ref{th:main}). Let us recall the definition of the discrete log-Gamma polymer \cite{log-gamma}. Let $\Gamma_n$ be the square lattice where each vertex has a weight $w_u$ drawn independently from $\Gamma^{-1}(\gamma)$ and all edges have weight $1$. The log-Gamma polymer partition function is then:
\[
Z_n = \sum_{\pi: (1,1) \rightarrow (n,n)}\prod_{u \in \pi}w_u
\]
By Remark \ref{rm:mixed-model}, the lattice $G_n$ obtained in the mixed model is the same as in the log-Gamma polymer, except that edges have weights $-1$ instead of $1$. However, since every path $\pi: (1,1) \rightarrow (n,n)$ has an even number of edges, the partition functions of the two models will be equal.

\begin{theorem}\label{th:tracy--Widom}
Let $\sigma_{n}(A_n)$ be the smallest singular value of the adjacency matrix $A_n$ of the lattice $G_n$. Then for all $\gamma < \gamma^{\ast}$, where $\gamma^{\ast}$ is the unique positive root of the digamma function $\Psi$, we have:
\[
\Pp\left(\frac{-\log \sigma_n(A_n) - \bar{f}_{\gamma}n}{n^{1/3}} \leq r\right) \rightarrow F_{\mathrm{GUE}}\left(\left(\frac{\bar{g}_{\gamma}}{2}\right)^{-1/3}r\right)
\]
where $\bar{f}_{\gamma} = -2 \Psi(\gamma / 2)$, $F_{GUE}$ is the GUE Tracy-Widom distribution function and $\bar{g}_{\gamma} = - 2 \Psi''(\gamma / 2)$.
\end{theorem}

\begin{proof}
By Theorem 2.1 of \cite{jeremy}, we have for any $\gamma > 0$:
\begin{equation}\label{eq-gue}
\Pp\left(\frac{\log Z_n - \bar{f}_{\gamma}n}{n^{1/3}} \leq r\right) \rightarrow F_{\mathrm{GUE}}\left(\left(\frac{\bar{g}_{\gamma}}{2}\right)^{-1/3}r\right)
\end{equation}
For $\gamma < \gamma^{\ast}$, the assumptions of Theorem \ref{th:mixed-topk} for $k=1$ are satisfied, so the random variable $n^{-1/3}\abs{\log\sigma_1(A_{n}^{-1}) - \log \abs{Z_n}}$ converges to $0$ in probability. The proof follows by noting that $\sigma_{n}(A_n)=\frac{1}{\sigma_{1}(A_n^{-1})}$.
\end{proof}

Theorem \ref{th:main} follows from Theorem \ref{th:tracy--Widom} by invoking Lemma \ref{lm:dual} and noting that in the mixed model, the dual graph corresponds exactly to the lattice $G_n$ from Theorem \ref{th:tracy--Widom}.

\begin{remark}\label{rm:iid-model-polymer}
Note that the i.i.d. model also corresponds to a polymer model on the lattice, where the vertex and edge weights are independent given as in Remark \ref{rm:iid-model} and the weight of a path is the product of the weights of all edges and vertices it contains. By Proposition \ref{prop:only-edges}, one can consider a model with weights only on the edges which in this case satisfies $\E X_e=0$. If one could establish Proposition \ref{prop:intersect-aux} in the $\E X_e=0$ case and a result analogous to Theorem 2.1 of \cite{jeremy} for such an i.i.d. polymer, these would imply that an analogue of Theorem \ref{th:tracy--Widom}, and therefore Theorem \ref{th:main}, holds also for the i.i.d. model.
\end{remark}

\section{Appendix: combinatorial results on non-intersecting paths}\label{sec:comb}

The goal of this section is to prove Lemma \ref{lm:comb-everything}. 

A naive approach to the proof would be as follows. Given sets $S,T$ and a nonintersecting tuple $\pi$ connecting $S$ to $T$, we would like to connect each point in $S$ to a point in $S_0$, and likewise for $T$ and $T_0$, with some path so as to obtain a nonintersecting tuple connecting $S_0$ to $T_0$. However, it is easy to give examples where this cannot be done, e.g. when one of the points in $S$ is cut off from the origin by some paths in $\pi$. Therefore, a more careful approach is needed.

We start with an informal outline of the proof. Observe that if $\pi$ contains intersecting paths, one can switch them to make them noncrossing with each other (see Figure \ref{fig:uncross}). One can then push the paths away from each other to remove the intersections (see Figure \ref{fig:pushaway}). However, in the process we can lose an unbounded number of edges (e.g. consider two zigzag paths touching each other at each corner). The key idea to prevent this is to connect every missing vertex from $S_0$ to a vertex from $S$ using a path $\tau$ which is simple, i.e. goes only right and then up. All modifications of paths will occur in a neighborhood of $\tau$ of fixed radius. Since $\tau$ is simple, every path in such a neighborhood can make only a bounded, independent of $n$, number of turns, which in turn guarantees that we always remove only a bounded number of edges.

The above approach is formalized as follows. Lemma \ref{lm:uncrossing} states that every tuple of paths can be uncrossed. Lemma \ref{lm:pushing} states that noncrossing paths can be pushed away to make them nonintersecting while losing only a bounded number of edges. For technical reasons this requires ensuring that paths intersect properly (Definition \ref{def:proper}), which is handled by Lemma \ref{lm:clean}. Finally, Lemma \ref{lm:intersect-main-comb} combines the previous lemmas to inductively connect missing points from $S_0$ to points from $S\backslash S_0$. The lemma requires that there are no points in $S$ too close to the boundary, which is ensured by Lemma \ref{lm:bottom}. The section ends with the proof of Lemma \ref{lm:comb-everything}.

For a path $\tau$ we let $N(\tau)$ be the neighborhood of $\tau$ of radius $6k$, i.e. the set of all vertices within $\ell_1$ distance at most $6k$ of some vertex on $\tau$. For a sequence $S$ of initial vertices let $\ordST(S) := \min\{i \ \vert \ (1,i) \notin S \}$ be the $y$ coordinate of the lowest vertex $(1,i)$ not contained in $S$ and likewise for a sequence $T$ of terminal vertices let $\ordST(T) := \max\{i \ \vert \ (n,i) \notin T \}$. 

\begin{definition}\label{def:proper}
For two paths, $\pi_1$ from $s$ to $t$ and $\pi_2$ from $s'$ to $t'$, such that $\pi_1 \cap \pi_2 \neq \emptyset$ we say that they {\bf intersect properly} if $s',t' \notin \pi_1$ and $s,t \notin \pi_2$.  
\end{definition}

We define a partial order on the vertices of the lattice, letting $(x,y) > (x',y')$ if $y > y'$ and $x < x'$. Consider two paths $\pi_1, \pi_2$. If there exist vertices $x \in \pi_1, y \in \pi_2$ such that $x > y$ and for all $x' \in \pi_1, y' \in \pi_2$ if $x'$ and $y'$ are comparable, then also $x' > y'$, we will write $\pi_1 \succ \pi_2$.  We write $\pi_1 \succeq \pi_2$ if $\pi_1 \succ \pi_2$ or $\pi_1 = \pi_2$.

\begin{definition}
A crossing intersection between $\pi_1$ and $\pi_2$ is a maximal connected subset of vertices $C \subseteq \pi_1 \cap \pi_2$ such that $\pi_1$ enters $C$ from the left and exits from the right and $\pi_2$ enters $C$ from the bottom and exits up. We say that $\pi_1$ and $\pi_2$ are {\bf crossing} if $\pi_1 \cap \pi_2$ contains at least one crossing intersection. We call a tuple of paths noncrossing if none of the paths in the tuple are crossing.
\end{definition}

\begin{remark}\label{rm:comparable}
If $\pi_1$ and $\pi_2$ intersect properly and are not crossing, then $\pi_1 \succ \pi_2$ or vice versa.
\end{remark}

\begin{lemma}\label{lm:uncrossing}
Fix $\sigma, S,T$ and consider a tuple of paths $\pi=(\pi_1,\dots,\pi_k)$, where $\pi_i$ connects $s_i$ to $t_{\sigma(i)}$. There exists a permutation $\sigma'$ and a tuple $\pi'=(\pi'_1,\dots,\pi'_k)$, where $\pi'_i$ connects $s_i$ to $t_{\sigma'(i)}$, such that $\pi'$ uses exactly the same multiset edges as $\pi$ and paths in $\pi'$ are noncrossing. Moreover, there is no pair $j < j'$ such that $\pi_{j'} \prec \pi_{j}$. We will say that the tuple $\pi'$ is obtained by {\bf uncrossing} $\pi$.
\end{lemma}

\begin{proof}
We will prove the following statement, from which the lemma follows. Consider any sequences $S,T$ of vertices of length $k$ and a multiset of edges $P$ such that for every vertex not in $S \cup T$ the number of ingoing edges is equal to the number of outgoing edges and the total number of edges outgoing from $S$ minus the number of edges ingoing into $S$ is $k$. We claim that there exists a tuple of paths $\pi=(\pi_1,\dots,\pi_k)$ and a permutation $\sigma$ such that $\pi_j$ connects $s_j$ to $t_{\sigma(j)}$, the tuple $\pi$ uses exactly the edges from $P$ and the paths $\pi_j$ are noncrossing.

We proceed by induction with respect to $k$. For $k=1$ the statement is obvious. Suppose we have already proved the lemma for $k-1$. For $i=2,\dots,2n$ let $L_i$ be the line $x+y=i$. In each strip between $L_i$ and  $L_{i+1}$ we list the edges from $P$ from the southeasternmost one to the northwesternmost one. We assign each edge a label equal to the number of edges preceding it (multiple copies of an edge are assigned consecutive numbers). Consider the northeasternmost vertex $s \in S$ among all vertices from $x \in S$ such that the diagonal line started at $x$ and going southeast intersects no other vertices used by the edges from $P$. There exists a path $\rho$ starting at $s$ and ending at some $t \in T$ which uses only edges labelled $0$. Indeed, suppose this was not the case and follow the edges labelled $0$ starting from $s$. Take the first vertex $x$ such that there is no outgoing edge from $x$ with label $0$. This implies there is a vertex $y$ southeast of $x$ with an outgoing edge labelled $0$. The vertex $y$ must belong to $S$ and is to the northeast of $s$, contradicting the choice of $s$.

After removing $\rho$, by the inductive hypothesis the remaining edges can be arranged into a noncrossing tuple of paths $\widetilde{\pi} = (\widetilde{\pi}_2,\dots,\widetilde{\pi}_{k})$ and the tuple $\widetilde{\pi}$ can be reordered in such a way that there is no pair $j < j'$ such that $\widetilde{\pi}_{j'} \prec \widetilde{\pi}_{j}$. The tuple $\pi'=(\rho,\widetilde{\pi}'_1,\dots,\widetilde{\pi}'_{k-1})$ is also noncrossing since $\rho$ used only edges labelled $0$, so it cannot cross any of the paths $\widetilde{\pi}'_{i}$. For the same reason the tuple $\pi'$ will be ordered in the desired way.
\end{proof}

\begin{figure}[h!]
  \centering
\includegraphics[scale=1]{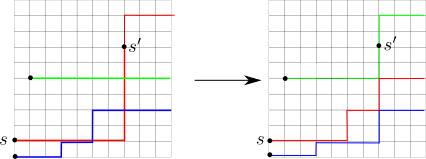}
  \caption{Switching paths to make them non-crossing. }
\label{fig:uncross}
\end{figure}

\begin{definition}\label{def:lowermost}
Let $\pi=(\pi_1, \dots,\pi_k)$ be a tuple of paths. Let $s,t$ be the starting and ending vertices of the path $\pi_j$ for some $j$. Let $P$ be the set of all paths that start at $s$, end at $t$ and use edges from the set $\cup_{j=1}^{k}\pi_j$. We define $\min(s,t;\pi)$ to be the smallest path, with respect to $\succ$, among all the paths in $P$.  
\end{definition}

Note that this is well defined since for any two paths $\rho,\rho' \in P$ there exists a path $\rho'' \in P$ such that $\rho \succeq \rho'', \rho' \succeq \rho''$. The path $\rho''$ can be obtained by uncrossing $\rho,\rho'$ (Lemma \ref{lm:uncrossing}) and taking the southeasternmost of the resulting paths. Since every set of paths in $P$ has a lower bound, there exists a unique smallest element.	


\begin{definition}
We say a path makes a turn at a vertex $v$ if it enters $v$ from the left and exits up or enters from below and exits right. We call a path $\tau$ {\bf simple} if it makes at most one turn.
\end{definition}

\begin{remark}\label{rm:ntau}
Let $\tau$ be a simple path. If a path is contained inside the neighborhood $N(\tau)$, the number of turns it makes is bounded by some constant $D$ depending only on $k$. This implies that the number of possible paths lying inside $N(\tau)$ is bounded by $D' \cdot n^{D''}$ for some constants $D',D''$ depending only on $k$.
\end{remark}

In the following lemmas we fix some $\sigma,S^{(0)},T^{(0)}$ and a tuple $\pi^{(0)} \in \Pi_{\sigma,S^{(0)},T^{(0)}}$. We also fix a simple path $\tau$ that starts at a vertex $(1,m)$ for some $m$. For any tuple $\pi=(\pi_1,\dots,\pi_k)$, a constant $C$ and any $0 \leq i \leq k$ we define the litions:
\begin{enumerate}[(1)]
\item \label{it:cond-nonintersect} for all $j \geq k-i+1$ the paths $\pi_{j}$ do not intersect any paths in $\pi$
\item \label{it:cond-succ}there are no $j,j'$ such that $j > j'$ and $\pi_{j} \prec \pi_{j'}$
\item \label{it:cond-bounded}$\abs{E(\pi^{(0)}) \backslash E(\pi)} \leq C$  
\item \label{it:cond-5i}all the vertices at which paths in $\pi$ intersect and all the edges in $E(\pi) \triangle E(\pi^{(0)})$ lie within distance $5(i+1)$ of $\tau$; also, the set $E(\pi) \triangle E(\pi^{(0)})$ is a union of at most $C$ paths  
\item \label{it:cond-1m} for all $j \leq m$ the vertex $(1,j)$ lies on exactly one path from $\pi$ and for all $j>m$ the vertex $(1,j)$ lies on at most one path from $\pi$
\end{enumerate}

\begin{lemma}\label{lm:clean}
Let $\pi \in \Pi_{\sigma,S,T}$ for some $\sigma,S,T$ be a noncrossing tuple that satisfies Conditions (\ref{it:cond-nonintersect})-(\ref{it:cond-1m}) for some $i$ and constant $C$. Then there exists a noncrossing tuple $\pi' \in \Pi_{\sigma',S',T'}$ for some $\sigma',S',T'$ that also satisfies Conditions (\ref{it:cond-nonintersect})-(\ref{it:cond-1m}) for $i$ and some constant $C'$ depending only on $C$ and $k$  such that all the paths in $\pi'$ intersect properly.
\end{lemma}

\begin{proof}
Let $(s_j,t_\sigma(j))$ be the pair of starting and terminal vertices of the path $\pi_j$. Suppose there are some $j \neq j'$ such that $s_{j'}$ lies on $\pi_j$ (the case of $t_{\sigma(j')}$ is dealt with analogously). Note that by Condition (\ref{it:cond-1m}) $s_{j'} \neq (1,y)$ for any $y$. Let $s'$ be defined as the first vertex on $\pi_{j'}$ which does not belong to any path other than $\pi_{j'}$. If $s'$ exists, let $\pi'_{j'}$ be the subpath of $\pi_{j'}$ started at $s'$. We replace $s_{j'}$ with $s'$ and $\pi_{j'}$ with $\pi'_{j'}$. If there is no $s'$, we replace both $s_{j'}$ and $t_{\sigma(j')}$ with any vertex in $N(\tau)$ southeast of $\pi_{j'}$ which does not belong to any path and take $\pi'_{j'}$ to be a path of length zero. Note that since $\pi_{j'}$ intersected some other path, by Condition (\ref{it:cond-nonintersect}) we must have $j' \geq i$, so to maintain Condition (\ref{it:cond-succ}) we can reorder the paths $\pi_j$ for $j \geq i$ and this does not violate Condition (\ref{it:cond-nonintersect}).

\begin{figure}[h!]
  \centering
\includegraphics[scale=1]{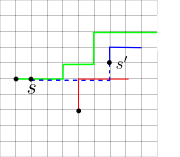}
  \caption{Lemma \ref{lm:clean}.}
\label{fig:step2}
\end{figure}

We claim that the resulting tuple of paths $\pi'$ uses all but boundedly many edges of $\pi_{j'}$, independent of $n$. Let:
\[
s_{j'} = x_0, x_1, \dots, x_l = s'
\] 
be the sequence of vertices on $\pi_{j'}$ between $s_{j'}$ and $s'$. Suppose that we lose an edge $(x_{n},x_{n+1})$ used by $\pi_{j'}$. Then for some path $\rho\in \pi$ we have $x_n \in \pi_{j'} \cap \rho$ and $x_{n+1} \notin \rho$ (otherwise the edge is still contained in $\rho$). All the paths are noncrossing, so, apart possibly of the first edge of $\pi_{j'}$, either $\pi_{j'}$ or $\rho$ makes a turn at $x_n$. By Condition (\ref{it:cond-5i}) the intersection at $x_n$, hence also the turn, happens inside $N(\tau)$. This implies that the total number of edges lost is at most the number of turns made by any path inside $N(\tau)$, which by Remark \ref{rm:ntau} is bounded by some constant $D$ independent of $n$, plus one. 

Since we have removed at most $D+1$ edges, Condition (\ref{it:cond-bounded}) still holds with the constant $C$ increased by $D+1$. The remaining conditions are clearly unchanged and the resulting tuple is still noncrossing. We can repeat the above step and after at most $2k$ such steps all the paths will intersect properly. The Condition (\ref{it:cond-bounded}) will be satisfied with $C' = C + 2k(D+1)$.
\end{proof}

\begin{lemma}\label{lm:pushing}
Let $\pi \in \Pi_{\sigma,S,T}$ for some $\sigma,S,T$ be a noncrossing tuple that satisfies Conditions (\ref{it:cond-nonintersect})-(\ref{it:cond-1m}) for some $i$ and constant $C$ and all intersecting paths in $\pi$ intersect properly. Then there exists a tuple $\pi' \in \Pi_{\sigma,S,T}$ that satisfies Conditions (\ref{it:cond-nonintersect})-(\ref{it:cond-1m}) for $i+1$ and some constant $C'$ depending only on $C$ and $k$.
\end{lemma}

\begin{proof}
To pass from $i$ to $i+1$ we need to ensure Condition (\ref{it:cond-nonintersect}) for $i$. Suppose that $\pi_{k-i}$ intersects some path $\pi_j$ with $j < k-i$. For every such $j$, we do the following. Informally, we want to push the path $\pi_j$ right and down at the intersections to remove the intersections, see Figure \ref{fig:pushaway}. To make this notion precise, let $P\pi_{k-i}$ denote the path $\pi_{k-i}$ with the edges touching the bottom or right boundary are removed and the remaining edges shifted one vertex down and one vertex to the right. We let $\pi'_j := \min(s_j,t_{\sigma(j)};(\pi_j, P\pi_{k-i}))$, defined as in Definition \ref{def:lowermost}. The path $\pi'_j$  starts and ends at the same vertices as $\pi_j$. We claim that:
\begin{itemize}
\item $\pi'_j$ and $\pi_{k-i}$ do not intersect
\item $\pi'_{j} \cup \pi_{k-i} $ uses all but boundedly many edges of $\pi_j$, independent of $n$
\end{itemize}

The paths $\pi_{k-i}$ and $\pi_{j}$ intersect properly and are noncrossing, which by Remark \ref{rm:comparable} and Condition (\ref{it:cond-succ}) implies that $\pi_{k-i} \succ \pi_{j}$. Assume by contradiction that $\pi'_j \cap \pi_{k-i} \neq \emptyset$ and take $u=(x,y)$ to be the first vertex in $\pi'_j \cap \pi_{k-i}$. The paths go together from $u$ to some vertex $v=(x',y')$, possibly $v=u$, at which they exit to different vertices. We claim that $\pi_{k-i}$ enters $u$ from the left and $\pi'_{j}$ enters $u$ from the bottom. If not, the path $\pi'_{j}$ would contain a vertex $w$ to the left of $u$ and $\pi_{k-i}$ a vertex $w'$ below $u$, so $w > w'$. Since $\pi_{k-i} \succ \pi_j$, this would imply $(w,u) \notin \pi_j$, so the edge $(w,u)$ must come from $P\pi_{k-i}$, which is clearly impossible. In the same fashion we prove that leaves $v$ up and $\pi'_{j}$ leaves $v$ to the right. Let $\rho$ be the subpath of $\pi'_{j}$ connecting $(x,y-1)$ to $(x'+1,y')$. Let $P\rho$ be $\rho$ shifted one vertex down and right, so clearly $P\rho \subseteq P\pi_{k-i}$. By considering the path $\pi_{j}$ with the path connecting $(x,y)$ to $(x',y')$ replaced by $P\rho$, we reach a contradiction with the minimality of $\pi'_j := \min(s_j,t_{\sigma(j)};(\pi_j, P\pi_{k-i}))$. 



\begin{figure}[h!]
  \centering
\includegraphics[scale=1]{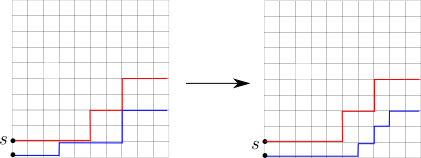}
  \caption{Pushing paths away to make them non-intersecting.}
\label{fig:pushaway}
\end{figure}
We now show that $\pi'_j \cup \pi_{k-i}$ use all but boundedly many edges of $\pi_j$. We first claim the only edges lost are the ones between some $u \in \pi_j \backslash \pi_{k-i}$ and its neighbor $v \in \pi_j \cap \pi_{k-i}$. Suppose there is a horizontal edge $(u,v) \in \pi_j \backslash (\pi_{k-i} \cup \pi'_{j})$ such that $u,v$ do not belong to $\pi_{k-i}$. Let $u',v'$ be the vertices $u,v$ shifted one right and down. It is easy to see that none of the vertices $u,v,u',v'$ are available for the path $\pi'_{j}$. Since $\pi_{j}\succ\pi'_{j}$, in order to cross the northwest-southeast diagonal line through $u$ the path $\pi'_{j}$ must pass through the set $\{x : x < u\}$ using a vertex belonging to $P\pi_{k-i}$. This would imply the existence of a vertex $w \in \pi_{k-i}$ such that $w < u$, contradicting $\pi_{k-i} \succ \pi_{j}$. The case of a vertical edge is handled in the same way.

Now, consider an edge lost between a vertex $u \in \pi_j \backslash \pi_{k-i}$ and its neighbor $v \in \pi_j \cap \pi_{k-i}$. Since the paths are noncrossing, one of the paths makes a turn at $u$ or $v$. Thus, the number of edges lost is bounded by twice the number of turns made by any path inside $N(\tau)$, which is bounded by some constant $D$ by Remark \ref{rm:ntau}. By the same reasoning, the set $E(\pi'_{j}) \backslash E(\pi_{j})$ is a union of a set of paths of size bounded by $2D$, since it consists of path segments starting or ending at vertices where one of the paths makes a turn.

We first check that by replacing $\pi_j$ with $\pi'_j$ we have not violated Condition (\ref{it:cond-nonintersect}). Suppose that for some $l \geq k-i+1$ the paths $\pi'_j$ and $\pi_{l}$ intersect at some vertex $x$. We must have $x \notin \pi_j$, as otherwise $\pi_j \cap \pi_{l} \neq \emptyset$, contradicting Condition (\ref{it:cond-nonintersect}). We must then have $x \in P\pi_{k-i}$, which means there is a vertex $x'\in \pi_{k-i}$ such that $x' > x$. Since the paths are noncrossing, by Remark \ref{rm:comparable} this would imply $\pi_{k-i} \succ \pi_{l}$, contradicting Condition (\ref{it:cond-succ}). 

We check that by replacing $\pi_j$ with $\pi'_j$ we have not violated Condition (\ref{it:cond-succ}). Suppose to the contrary that $\pi'_{j} \succ \pi_{l}$ for some $l > j$. This means that for some $x \in \pi'_{j}$ and $y \in \pi_{l}$ we have $x > y$. Since $\pi_j$ and $\pi'_j$ have the same starting and terminal vertices and $\pi_j \succeq \pi'_j$, the path $\pi_j$ must cross the region $\{z : z \geq x\}$, in particular, it contains a vertex $z$ such that $z \geq x > y$, proving that $\pi_j \succ \pi_{l}$ and contradicting Condition (\ref{it:cond-succ}). The proof that Condition (\ref{it:cond-succ}) holds for $l < j$ proceeds in the same fashion.

Note that after this procedure the path $\pi_{k-i}$ still intersects properly and is noncrossing with other paths that it intersects. We can thus repeat the above procedure for all $j < k-i$ such that $\pi_{k-i}$ and $\pi_j$ intersect, obtaining a new tuple $\pi'$ in which $\pi_{k-i}$ does not intersect any paths. The Condition (\ref{it:cond-nonintersect}) holds now for $i+1$. Condition (\ref{it:cond-bounded}) holds now for $i+1$ since we have removed at most $2kD$ edges and likewise for Condition (\ref{it:cond-5i}). Note that every path has been moved by distance at most $1$ and all the edges added or removed were within distance $5i$ of $\tau$. Since Condition (\ref{it:cond-5i}) held for $i$, this implies that now it holds for $i+1$. The resulting tuple $\pi'$ thus satisfies all the required conditions for $i+1$.
\end{proof}

\begin{lemma}\label{lm:intersect-main-comb}
Fix $\sigma,S,T$, with either $S \neq S_0$ or $T \neq T_0$ up to reordering. Suppose that $S \backslash S_0$ does not contain any vertices closer than $k+1$ to the bottom boundary and $T \backslash T_0$ does not contain any vertices closer than $k+1$ to the top boundary. There exists a simple path $\tau$ with the following properties. For any non-intersecting tuple $\pi \in \Pi^{n.i.}_{S,T,\sigma}$ there exist $\sigma', S', T'$ and a non-intersecting tuple $\pi' \in \Pi^{n.i.}_{S',T',\sigma'}$ such that $\ordST(S') > \ordST(S)$ (or, if $S=S_0$, $\ordST(T') < \ordST(T)$), $\abs{E(\pi) \backslash E(\pi')}$ is bounded independently of $n$ and $E(\pi)\triangle E(\pi')$ is a union of paths whose number is bounded independently of $n$ and which all lie within $N(\tau)$.
\end{lemma}

\begin{proof}
We treat the case $S \neq S_0$, as the case of $T$ is analogous.

Let $(1,m)$ be the vertex in $S_0 \backslash S$ with smallest $m$. Pick the path $\tau$ to be the one connecting $(1,m)$ to any vertex $s_j \in S \backslash S_0$ and such that it goes right and then up (possibly only up if $s_j=(1,t)$ for some $t$). This is possible since we assume that no vertices in $S$ lie within the first $k$ levels from the bottom. We modify the set $S$ by removing $s_j$ and inserting $(1,m)$. We replace the path $\pi_j$, connecting $s_j$ to $t_{\sigma(j)}$, with $\tau\cdot\pi_j$, the concatenation of $\tau$ and $\pi_j$, which now starts at $(1,m)$. In this way, we obtain a new tuple that, in general, may contain intersecting paths.

We will construct the desired tuple $\pi'$ using an inductive procedure. We let $\pi^{(0)}:=\pi$. At step $0 \leq i \leq k$ we will have sets $S^{(i)},T^{(i)}$ and a possibly intersecting tuple $\pi^{(i)}=(\pi^{(i)}_1,\dots,\pi^{(i)}_k)$ connecting $S^{(i)}$ to $T^{(i)}$ and satisfying Conditions (\ref{it:cond-nonintersect})-(\ref{it:cond-1m}) for $i$. Applying Lemma \ref{lm:clean} and \ref{lm:pushing} will allow us to construct $\pi^{(i+1)}$ satisfying the conditions for $i+1$. 

To obtain $\pi^{(1)}$, we do the following. Condition (\ref{it:cond-1m}) clearly holds for vertices $(1,l)$ with $l < m$ since the initial tuple $\pi$ was nonintersecting. To ensure that it holds also for $(1,m)$, let $j$ be such that $(1,m)$ is the starting vertex of the path $\pi^{(0)}_{j}$ and suppose that $(1,m)$ also lies on a path $\pi^{(0)}_{j'}$ with $j \neq j'$. Because $\pi$ was nonintersecting, the path $\pi^{(0)}_{j'}$ must start at $(1,m-1)$. Recall that the path $\tau$, which is an initial segment of $\pi^{(0)}_{j}$, connects $(1,m)$ to the vertex $s_j$. 

If $s_j\neq(1,l)$ for any $l$, we switch paths so that the new path $\widetilde{\pi}^{(0)}_{j'}$ starts with the path $(1,m-1)\rightarrow (2,m-1) \rightarrow (2,m)$ and then follows the edges of $\pi^{(0)}_{j}$, while $\widetilde{\pi}^{(0)}_{j}$ starts at $(1,m)$ and follows the edges of $\pi^{(0)}_{j'}$. This is shown in Figure \ref{fig:special-case}.

If $s_j=(1,l)$ for some $l > m$, consider the endpoint $t$ of the path $\pi^{(0)}_{j'}$. If $t=(1,l')$ for some $l' \geq m$, we must have $l' < l$ because in the tuple $\pi$ the paths $\pi^{(0)}_{j'}$ and the subpath of $\pi^{(0)}_{j}$ starting at $s_j$ were nonintersecting. We replace the vertex $t$ by $(1,m-1)$ and replace $\pi^{(0)}_{j'}$ with a path of length zero. If $t \neq (1,l')$, let $(2,a)$ be the first vertex on $\pi^{(0)}_{j'}$ in the second column. We replace the segment of $\pi^{(0)}_{j'}$ connecting $(1,m-1)$ to $(2,a)$ with a path that first goes right and then up to $(2,a)$. 

\begin{figure}[h!]
  \centering
\includegraphics[scale=1]{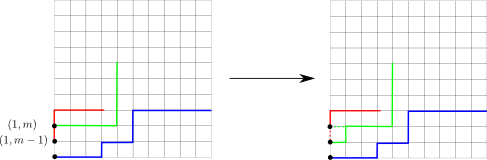}
  \caption{Switching paths $\pi^{(0)}_{j}$ and $\pi^{(0)}_{j'}$.}
\label{fig:special-case}
\end{figure}

We then apply Lemma \ref{lm:uncrossing} to make the tuple noncrossing and ensure that there are no $j' > j$  such that $\pi^{(0)}_{j'} \prec \pi^{(0)}_j$. The procedure above ensures that in all the cases considered the Conditions (\ref{it:cond-bounded}) and (\ref{it:cond-5i}) hold. In this way all the Conditions (\ref{it:cond-nonintersect})-(\ref{it:cond-1m}) for $i=0$ are satisfied.

To pass from $\pi^{(i)}$ to $\pi^{(i+1)}$, we apply the following steps. We first apply Lemma \ref{lm:uncrossing} to obtain a tuple $\pi^{(i)'}$ with all the paths non-crossing. Then, we apply Lemma \ref{lm:clean} to obtain a tuple $\pi^{(i)''}$ in which all the paths intersect properly. We can then apply Lemma \ref{lm:pushing}, obtaining a tuple $\pi^{(i)'''}$ which satisfies Conditions (\ref{it:cond-nonintersect})-(\ref{it:cond-1m}) for $i+1$. We can then put $\pi^{(i+1)}:=\pi^{(i)'''}$.

After performing the inductive procedure, we obtain a tuple $\pi^{(k)}$ that satisfies all the required conditions. We put $S':=S^{(k)}, T':=T^{(k)}, \pi':=\pi^{(k)}$. By Condition (\ref{it:cond-nonintersect}) the tuple is nonintersecting. By Condition (\ref{it:cond-bounded}) $\abs{E(\pi) \backslash E(\pi')}$ is bounded independently of $n$. By Condition (\ref{it:cond-5i}) $E(\pi)\triangle E(\pi')$ is a union of paths whose number is bounded independently of $n$ and which all lie within $N(\tau)$. Finally, since the vertices $(1,l)$ for $l\leq m$ were never removed, we have $\ordST(S') > \ordST(S)$.
\end{proof}

\begin{lemma}\label{lm:bottom}
For any $\sigma, S, T$ and a tuple $\pi \in \Pi^{n.i.}_{S,T,\sigma}$ there exist $\sigma', S', T'$ and a tuple $\pi' \in \Pi^{n.i.}_{S',T',\sigma'}$ such that $S' \backslash S_0$ does not contain any vertices closer than $k+1$ to the bottom boundary, $T' \backslash T_0$ does not contain any vertices closer than $k+1$ to the top boundary, $\abs{E(\pi) \backslash E(\pi')}$ is bounded independently of $n$ and $E(\pi)\triangle E(\pi')$ is a union of paths whose number is bounded independently of $n$ and which all lie within distance $k$ of the boundary.
\end{lemma}

\begin{proof}
We prove the statement for the part concerning $S$, the part concerning $T$ is analogous. Let $S_l$ be the sequence of vertices from $S \backslash S_0$ located at some level $1 \leq l \leq k$ from the bottom boundary. Take the smallest $l$ so that $S_l$ is nonempty. Let $S_l=(s_{i_{1}},\dots,s_{i_{m}})$. We take the vertex $(1,l)$ and connect it with the unique path $\tau$ it to the rightmost vertex $s_{i_{m}} \in S_l$. For $j=1,\dots,m$, let $\pi_{i_{j}}$ be the path contained in $\pi$ starting from $s_{i_{j}}$. If $m>1$, for $1\leq j<m$ we replace the vertex $s_{i_{j}}$ with $s'_{i_{j}}$, defined as the first vertex on level $l+1$ which lies along $\pi_{i_{j}}$. If there is no such vertex, we replace the pair $(s_{i_{j}},t_{\sigma(i_{j})})$ with any vertex at least $k+1$ away from the boundary which does not belong to any path and modify $\pi_{i_{j}}$ to be a path of length zero.

If $(1,l) \in S$, the path $\tau$ will intersect the path $\rho$ belonging to $\pi$ which starts at $(1,l)$. In that case, we retain $(1,l)$ in $S$ and insert a new starting vertex $s'$, defined as the first vertex on level $l+1$ which lies along $\rho$. If there is no such vertex, we deal with this case in the same way as in the last case in the previous paragraph.

The tuple $\pi'$ thus obtained may contain intersecting paths, coming from intersections of $\tau$ with paths starting at vertices $(1,i)$ with $i<l$. We deal with these intersections by applying Lemma \ref{lm:clean} and \ref{lm:pushing}, exactly as in the proof of Lemma \ref{lm:intersect-main-comb}.

\begin{figure}[h!]
  \centering
\includegraphics[scale=1]{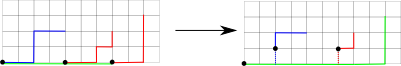}
  \caption{Removing the vertices at level $l$ for $l=1$. The path $\tau$ is shown in green.}
\label{fig:bottom}
\end{figure}

Overall, the procedure performed above is illustrated in Figure \ref{fig:bottom} for $l=1$. Note that the new tuple $\pi'$ uses all the edges contained in $\pi$, apart possibly from the vertical edges connecting levels $l$ to $l+1$, lost when replacing the vertices $s_{i_{j}}$ with $s'_{i_{j}}$, and edges removed in the pushing process. We have added the edges contained in $\tau$, plus possibly other edges during the switching and pushing part, all of which lie within distance $k$ to the boundary.
\end{proof}

\begin{proof}[Proof of Lemma \ref{lm:comb-everything}]
Given $\sigma, S, T$ and $\pi \in \Pi^{n.i.}_{S,T,\sigma}$, we first apply Lemma \ref{lm:bottom} to obtain $\sigma', S', T'$ and $\pi' \in \Pi^{n.i.}_{S',T',\sigma'}$ such that $\abs{E(\pi) \backslash E(\pi')}$ is bounded independently of $n$, the set $E(\pi) \triangle E(\pi')$ is a union of a bounded number of paths contained within distance at most $k$ from the bottom or top boundary, $S' \backslash S_0$ does not contain any vertices closer than $k+1$ to the bottom boundary and $T' \backslash T_0$ does not contain any vertices closer than $k+1$ to the top boundary. We are now in position to apply Lemma \ref{lm:intersect-main-comb} to obtain a tuple $\pi''$ and $\sigma'',S'',T''$ such that either $\ordST(S'') > \ordST(S)$ or $\ordST(T'') < \ordST(T)$, $E(\pi) \backslash E(\pi'')$ is bounded by some constant $C'$ independent of $n$ and $E(\pi) \triangle E(\pi'')$ is a union of at most $C'$ paths lying inside $N(\tau)\cup N_{b}$, where $\tau$ is some simple path, $N(\tau)$ is the neighborhood of $\tau$ of radius $6k$ and $N_{b}$ is the neighborhood of the boundary of radius $k$.

By repeating the above procedure, in each step increasing $\ordST(S)$ or decreasing $\ordST(T)$, after $m \leq 2k$ steps we arrive at sets $S,T$ with $\ordST(S)=k+1$ and $\ordST(T)=n-k$, which actually implies that $S=S_0$ and $T=T_0$, thus completing the proof. Let $\{\tau_i\}_{i=1}^{m}$ be the set of simple paths $\tau$ appearing at repeated applications of Lemma \ref{lm:intersect-main-comb}. We take $\mathcal{P}$ to be the set of all possible paths contained in $\cup_{i=1}^{m}N(\tau_{i})\cup N_{b}$. By Remark \ref{rm:ntau}, the set $\mathcal{P}$ has size at most $D' \cdot n^{D''}$ for some constants $D',D''$ depending only on $k$. The proof is finished by taking $C := \max\{D,D'',2kC'\}$.
\end{proof}

\bibliography{biblio}{}
\bibliographystyle{amsalpha}

\bigskip
\noindent
 Marcin Kotowski \\
 Faculty of Mathematics, Informatics, and Mechanics, University of Warsaw  \\
 ul. Banacha 2, 02-097 Warszawa \\
\noindent
{E-mail:} {\tt mk249015@mimuw.edu.pl} \\
\noindent
\url{http://www.math.toronto.edu/~marcin/}

\bigskip
\noindent
 B\'{a}lint Virag \\
 Department of Mathematics, University of Toronto \\
 Bahen Centre, 40 St. George St., Toronto, Ontario \\
 CANADA M5S 2E4
 \\
\noindent
{E-mail:} {\tt balint@math.toronto.edu} \\
\noindent
\url{http://www.math.toronto.edu/~balint/}

\end{document}